\documentclass[12pt]{article}

\usepackage[dvipdf]{graphicx}
\usepackage[latin1]{inputenc}
\usepackage{latexsym}
\usepackage{amsmath,amssymb,amsfonts,amsthm}
\usepackage{xcolor}
\usepackage{mathrsfs}
\usepackage{bbm}
\usepackage{bm}
\usepackage{enumitem}

\usepackage[numbers,comma,sort]{natbib}

\definecolor{db}{RGB}{0, 0, 130}
\definecolor{wildstrawberry}{rgb}{1.0, 0.26, 0.64}

\usepackage[colorlinks=true,citecolor=db,linkcolor=black,urlcolor=blue,pdfstartview=FitH]{hyperref}

\usepackage{pdfsync}

\definecolor{rp}{rgb}{0.25, 0, 0.75}
\definecolor{dg}{rgb}{0, 0.5, 0}

\newcommand{\R}{\mathbb{R}}

\newcommand{\N}{\mathbb{N}}

\newcommand{\mm}[1]{{\color{red}#1}}

\makeatletter
\newcommand{\customlabel}[2]{%
   \protected@write \@auxout {}{\string\newlabel {#1}{{#2}{\thepage}{#2}{#1}{}}}%
   \hypertarget{#1}{#2\hspace{-0.14cm}}
}
\makeatother

\numberwithin{equation}{section}

\newtheorem{theorem}{Theorem}[section]

\newtheorem{definition}[theorem]{Definition}

\newtheorem{lemma}[theorem]{Lemma}

\title{ \Large{\textbf{From  Hamiltonian Systems to  Compressible  Euler Equation driven by \mm{additive  H\"older}  noise}}}

\author{Jesus M. Correa\thanks{Departamento de Matem\'{a}tica, Universidade Estadual de
Campinas, Brazil. Email: \textsl{j209372@dac.unicamp.br}},Juan David Londoño Acevedo
\thanks{ Escuela de Ciencias Básicas, Tecnología e Ingeniería - ECBTI
Universidad Nacional Abierta y a Distancia-CCAV Sahagún
Colombia \textsl{ e-mail: juan.londono@unad.edu.co}}
,  Christian Olivera\thanks{Departamento de Matem\'{a}tica, Universidade Estadual de
Campinas, Brazil. Email: \textsl{colivera@ime.unicamp.br}.}}

\date{}

\begin{document}

\maketitle

\begin{abstract}
We derive stochastic compressible  Euler Equation from a Hamiltonian microscopic dynamics.
We consider  systems of interacting particles with H\"older noise  and  potential whose range is large in comparison with the typical distance between neighbouring particles. It is shown that the empirical measures associated to the position and velocity of the system converge to the solutions of   compressible Euler equations driven  by  additive H\"older path(noise), in the limit as the particle number tends to infinity, for a suitable scaling of the interactions. 
Furthermore, explicit  rates for the convergence  are
 obtained in Besov and Triebel-Lizorkin spaces.
Our proof  is based on  the It\^o-Wentzell-Kunita formula for  Young integral. 
\end{abstract}

\date{}

\maketitle


\vspace{0.3cm} \noindent {\bf MSC2010 subject classification:} 60H15, 
 35R60, 60H30, 35Q31
. 

%
%
%
%
%
%
%

Data sharing not applicable to this article as no datasets were generated or analysed during the current study
\section{Introduction}
This paper is aimed to derive macroscopic continuous models characterizing the
asymptotic behavior of interacting particle systems as the number of particles goes to
infinity. More precisely, we analyse the evolution of an indistinguishable $N-$ point particle
system given by

\begin{align}\label{xk}
d^{2}X_{t}^{k,N}=-\frac{1}{N}\sum_{l=1}^{N}\nabla\phi_{N}\left(X_{t}^{k,N}-X_{t}^{l,N}\right)dt
+ \sigma\left(t,X_{t}^{k,N} \right)   dY_{t} 
\end{align}

here  $X^{i,N} $  and   $V^{i,N}=\frac{d X^{i,N}}{dt}$  denote the position and velocity of $i$-particle at time t,respectively,  
  $\{(Y_{t}^{q})_{t\in[0,T]}, \; q=1,\cdots,d\}$ is a  $\R^d$-valued $\alpha$-H\"older path with $\alpha > \frac{1}{2}$. 
Follows  \cite{Oes} the  interaction potential  $\phi_{N}$  is obtained from a function  $\phi_{1}$  by the scaling

 \begin{equation}\label{molli}
 \phi_{N}(x)=N^{\beta}\phi_{1}(N^{\beta/d}x),\   \beta\in(0,1), 
\end{equation}

where  $\phi_{1}$ is symmetric and
sufficiently smooth. The parameter  $\beta\in (0,1)$ describes how $\phi_{1}$ is rescaled for the total number $N$ of
individuals and expresses the so-called moderate interaction among the individuals;
see Oelschl\"ager \cite{Oes2}. On the other hand, $\beta = 0$ expresses an interaction of mean field type, whereas $\beta = 1$ generates the so-called nearest-neighbour interaction. We analyze the asymptotic behavior among many particles 
moderately interacting with each other.

Our aim is the study of the asymptotic as $N\rightarrow\infty$ of  the time evolution
of the whole system of all particles. Therefore, we investigate the empirical processes

\begin{eqnarray}\label{empp}
S_{t}^{N}:=\frac{1}{N}\sum_{k=1}^{N}\delta_{X_{t}^{k,N}},
\end{eqnarray}

\begin{eqnarray}\label{empp}
 V_{t}^{N}:=\frac{1}{N}\sum_{k=1}^{N}V_{t}^{k,N}\delta_{X_{t}^{k,N}}\qquad k=1,\cdots, N
\end{eqnarray}

where  $\delta_{a}$, denotes the
Dirac measure at $a$. The measures $S_{t}^{N}$ and $V_{t}^{N}$ determine the distribution
of the positions and the velocities in the $Nth$ system.

We shall show that $S_{t}^{N}$ and $V_{t}^{N}$ converge as $N\rightarrow\infty$ to solutions of the
continuity equation and the  Euler equation, respectively driven by $\alpha$-H\"older path. 
More precisely  we show the convergence to  the system 

\begin{align}\label{eq1}
\left\{
\begin{array}{lc}
d\varrho =-\operatorname{div}_{x}(\varrho\upsilon)dt
\\[0.3cm]
d(\varrho\upsilon_{q})=-\left(\nabla^{q}p+\operatorname{div}_{x}(\varrho\upsilon_{q}\upsilon)\right)dt 
+ \varrho \sigma_{q}  dY_{t}^{q}\qquad q=1,\cdots, d
\end{array}
\right.
\end{align}

 or equivalently

\bigskip
\begin{align}\label{eq2}
\left\{
\begin{array}{lc}
d\varrho =-\operatorname{div}_{x}(\varrho\upsilon)dt
\\[0.3cm]
d(\upsilon_{q})=-\left(\frac{1}{\varrho} \nabla^{q}p+\upsilon  \nabla  \upsilon_{q} \right) dt+    \sigma_{q}  dY_{t}^{q}\qquad q=1,\cdots, d
\end{array}
\right.
\end{align}

where the pressure is $p=\frac{1}{2}\varrho^{2}.$   

A source of much research in mathematical physics is the problem of rigorously
deriving the  compressible Euler equation in dimensions $d\geq 2$.
One of the main open problems in  mechanics fluids  is
the derivation of the  stochastic  equations of fluids, the so-called  stochastic Euler and
Navier-Stokes equations, from the microscopic Hamiltonian dynamics. 
The main purpose of this work is to make the above formal derivation completely
rigorous. More precisely, we will provide quantitative estimation between solutions 
and the empirical measures.  The proof is based 
 on It\^o-Wentzell-Kunita formula for  Young integral, see \cite{CatuognoCastrequini},  and 
on   the paper \cite{Oes} where the  same equation   without noise  is considered.  
The noise in equation (\ref{eq2}) has a very special form, compared to general abstract models of stochastic partial differential equations (SPDEs). Our aim is not an
abstract generality We   point out that in (\ref{xk}) we consider $\alpha$-H\"older path but we believe
which can be easily extended for infinite noises. 	Concerning 
 the problem  of well-posedness of (\ref{eq1})-(\ref{eq2})  remains open,  see \cite{Fei}, \cite{Chae}, \cite{Chae2}
  \cite{Lions} and  \cite{Majda} on  well-posedness results in the deterministic setting.

\paragraph{Related Works.}

In \cite{Correa} we consider the case that the  Hamiltonian system is perturbed by the Brownian motion, for  
related works in the deterministic setting see  \cite{Carillo}, \cite{Dil} and \cite{Franz}. For particle 
approximations for stochastic  incompressible Euler-Navier stokes equations see \cite{Maurelli}  and \cite{Kotelenez}.
We also mention that  \cite{Crisan} derive a class of rough geophysical fluid dynamics 
models as critical points of rough action functionals using the theory of controlled rough paths. 
Finally an excellent recent monograph \cite{Jabin} where
some additional information can be found.

\paragraph{Example of H\"older  noise.}

Let $(B_{t}) _{t\in [0, T]}$ be a fractional Brownian motion with Hurst index $H\in (1/2, 1)$. Recall that $B$ is a centered Gaussian process with covariance
\begin{equation*}
  \mathbf{E} B_{t} B_{s}= \frac{1}{2} \left( t ^{2H}+ s^{2H}-\vert t-s\vert ^{2H} \right) \mbox{ for every }s,t\in [0,T].
\end{equation*}

This process  is self-similar of index $H$,  and their  paths are H\"older continuous of order $\alpha \in (0,H)$.

\section{ Definitions,  preliminaries and hypothesis.}

\subsection{Young Integral.}

Fix a time interval $[0,T]$. Let $U$, $V$ and $W$ be Banach spaces. Given a path $\phi:[0,T]\to V$ and $s,t\in[0,T]$ we write $\phi_{st}=\phi_{t}-\phi_{s}$.
	\begin{definition}
		Let $0<\alpha\leq 1$. 
		$\mathcal{C}^{\alpha}([0,T];V)$ is the space of functions on $[0,T]$ taking values in $V$ such that the following $\alpha$-H\"older seminorm 
		\begin{equation*}
			\|\phi\|_{\alpha}:=\sup_{0\leq s<t\leq T}\frac{|\phi_{st}|}{|t-s|^{\alpha}}.
		\end{equation*}
		is finite.
	\end{definition}
	The space $\mathcal{C}^{\alpha}([0,T];V)$ is a Banach space with the norm
	$$
	|\phi|_{\alpha}=\|\phi\|_{\alpha}+\|\phi\|_{\infty},
	$$
	where, as usual, $\|\phi\|_{\infty}=\sup_{t\in[0,T]}|\phi(t)|$.\\
	
	We will denote by $L(U,V)$ the space of continuous linear mappings from $U$ to $V$, and
	$$
	\Delta_{T}:=\left\{(s,t)\in[0,T]^{2}:0\leq s\leq t\leq T\right\}.
	$$
	In this section we define the Young's integral $\int XdY$ when $Y\in C^{\alpha}([0,T];V)$ and $X\in C^{\beta}([0,T];L(V,W))$ with $\alpha+\beta>1$. We refer the reader to the paper \cite{Young} by L.C. Young.\\
	The map
	\begin{equation*}
		(X,Y)\mapsto\left( t\mapsto\int_{0}^{t}X_{s}dY_{s}\right)
	\end{equation*}
is bilinear and continuous from $C^{\beta}([0,T];L(V,W))\times C^{\alpha}([0,T];V)$ to $C^{\alpha}([0,T];W)$.

Finally, the cornerstone of this theory is the following Young-Loeve estimative, see Proposition 3 in \cite{Gubinelli}.
	\begin{theorem}
		Let $Y\in C^{\alpha}([0,T];V)$ and $X\in C^{\beta}([0,T];L(V,W))$ for some $\alpha,\beta\in(0,1]$ with $\alpha+\beta>1$. Then the limit
		\begin{equation*}
			\int_{0}^{t}X_{r}dY_{r}:=\lim_{|\pi|\to 0}\sum_{[r,s]\in\pi}X_{r}Y_{rs}
		\end{equation*}
		exists for every $t\in[0,T]$, where the limit is taken over any $\pi\in\mathcal{P}([0,t])$, and $\mathcal{P}([0,t])$ the set of all partitions $\pi$ of the interval $[0,t]$. This limit is called the {\fontfamily{qmr}\selectfont Young integral} of $X$ against $Y$, which moreover holds the following estimative
		\begin{equation}\label{Loeve}
			\left|\int_{s}^{t}X_{r}dY_{r}-X_{s}Y_{st}\right|_{W}\leq C_{\alpha+\beta}\| Y\|_{\alpha}\| X\|_{\beta}|t-s|^{\alpha+\beta}
		\end{equation}
		for all $(s,t)\in\Delta_{T}.$
	\end{theorem}
	We recall the independence of the Young integral with respect to the choice of the partitioning points. Let $r\in[s,\theta]$ denote an arbitrary point in the interval $[s,\theta]$. The Young integral of $X$ against $Y$ is equal to the limit
	\begin{equation*}
		\int_{0}^{t}X_{s}dY_{s}:=\lim_{|\pi|\to 0}\sum_{[s,\theta]\in\pi}X_{r}Y_{s\theta}
	\end{equation*}
	for any $t\in[0,T]$.\\
	As the next lemma shows, Young integration satisfies the classical integration by parts formula.
	
	\begin{lemma}
		Let $X\in C^{\alpha}$ and $Y\in C^{\beta}$ for some $\alpha, \beta\in(0,1]$ with $\alpha+\beta>1$. Then
		\begin{equation}\label{produ}
			X_{T}Y_{T}=X_{0}Y_{0}+\int_{0}^{T}X_{u}dY_{u}+\int_{0}^{T}Y_{u}dX_{u}.
		\end{equation}
	\end{lemma}
\begin{lemma}
	Let $X\in C^{\alpha}([0,T],V)$ and $f\in C^{1+\gamma}(V,W)$ for some $\alpha, \gamma\in(0,1]$, such that $\alpha(1+\gamma)>1$. Then $\int_{0}^{T}Df(X_{r})dX_{r}$ is well-defined Young integral, and
	\begin{equation}\label{Ito}
		f(X_{T})=f(X_{0})+\int_{0}^{T}Df(X_{r})dX_{r}.
	\end{equation}
\end{lemma}

	We finish this section with the generalized It\^o-Wentzell formula for the Young integral, see Theorem 3.1 of \cite{CatuognoCastrequini}.
	
\mm{
	\begin{theorem}
		Let $\alpha\in(\frac{1}{2},1]$, $Y\in C^{\alpha}([0,T];V)$ and $h:[0,T]\times U\to L(V,W)$ continuous and differentiable in $U$ such that
		\begin{enumerate}
			\item  $(t,x)\to D_{x}h_{t}(x)$ is continuous,
			\item $h\in C(U,C^{\beta}([0,T];L(V,W)))$ for some $\beta\in(\frac{1}{2},1]$.
		\end{enumerate}
		Let
		\[
			g_{t}(x)=g_{0}(x)+\int_{0}^{t}h_{s}(x)dY_{s}.		
		\]
		We assume that $g:[0,T]\times U\to W$ is twice differentiable in $U$, the functions $(t,x)\mapsto D_{x}^{2}g_{t}(x)$ are continuous and 
 $D_{x}g\in C(U,C^{\gamma}([0,T],L(V,W)))$ for some $\gamma\in(\frac{1}{2},1]$.  Then for any $X\in C^{\alpha}([0,T],U)$,
		\begin{equation}\label{Ito-W}
			g_{t}(X_{t})=g_{0}(X_{0})+\int_{0}^{t}h_{s}(X_{s})dY_{s}+\int_{0}^{t}D_{x}g_{s}(X_{s})dX_{s},
		\end{equation}
where  the integral $\int_{0}^{t}D_{x}g_{s}(X_{s})dX_{s}$ is a Young integral, and 
\begin{align*}
t\mapsto g_{t}(X_{t})\in C^{mim(\alpha \wedge \lambda )}([0,T],W).
\end{align*}
	\end{theorem}
}

\subsection{Space of functions }

Let us first define a dyadic partition of unity, as follows: we consider two 
$C_{0}^{\infty}(\R^{d})$-functions $\chi$ 
 and $\varphi$  which take values in $[0,1]$ and satisfy the following: there exists $\lambda \in (1,\sqrt 2)$ such that
\[
\mathrm{Supp}\; \chi=\left\{   |\xi |\leq\lambda  \right\} \qquad \text{and} \qquad \mathrm{Supp}  \ \varphi=\left\{  \frac{1}{\lambda} \leq|\xi|\leq\lambda  \right\}.
\]
Moreover, with the following notations,  
\[
\varphi_{-1}(\xi):=\chi(\xi), \qquad
\varphi_{i}(\xi):=\varphi(2^{-i}\xi), \ \text{for any } i\geq 0,
\]
the sequence $\{\varphi_i\}$ satisfies
\begin{align*}
& \mathrm{Supp} \ \varphi_{i}\cap \mathrm{Supp} \ \varphi_{j}=\emptyset	 \quad \text{ if } \ |i-j|>1,  \\
& \sum_{i\geq-1} \varphi_{i}(\xi)=1, \quad \text{ for any } 
x \in\R^{d}.
\end{align*}

 Take a fixed dyadic partition of unity $\{ \varphi_{i} \}$ with its inverse Fourier transforms 
$  \{  \check{\varphi}_{i}   \}$. 
For $u\in \mathcal{S}'(\R^{d})$
, the nonhomogeneous dyadic  blocks are defined  as
\[
\Delta_{i}\equiv 0, \quad \text{ if } i< -1, \qquad \text{and} \qquad 
\Delta_{i}u=\check{\varphi}_{i}\ast u, \quad \text{ if  }  i\geq -1.
\]
The partial sum of dyadic blocks is defined as a nonhomogeneous low frequency cut-off
operator:
\[
S_{j}:=\sum_{i\leq j-1} \Delta_{i}. 
\]

Having the theses smooth resolution of unity we are able 
to introduce the Triebel-Lizorkin and Besov Space.

\begin{definition}  Let $s\in \R$ and $1< p<\infty$. 

\begin{enumerate}

\item  If $1< q< \infty$, then

\[
F_{p,q}^{s}= \left\{  f\in \mathcal{S}^{\prime}(\R^{d}) : \    \big\| \| (2^{js} \Delta_{j}f)_{j \in \mathbb{Z} } \|_{l^q(\mathbb{Z})}  \big\|_{\mathit{L}^{p}(\R^{d})}< \infty \right\}.
\]

\item If $0< q\leq \infty$, then

\[
B_{p,q}^{s}= \left\{  f\in \mathcal{S}^{\prime}(\R^{d}) : \  \big\| (2^{js}\| \Delta_{j}f \|_{L^p} )_{j \in \mathbb{Z} } \big\|_{\mathit{l}^{q}(\mathbb{Z})}< \infty \right\}.
\]
\end{enumerate}

\end{definition}

The spaces $F_{p,q}^{s}$  and $B_{p,q}^{s}$are independent of the dyadic partition  chosen, see \cite{Triebel}. We denoted 
$E_{p,q}^{s}=F_{p,q}^{s}$ or $E_{p,q}^{s}=B_{p,q}^{s}$ with $q\geq 2$. We recall the Sobolev embedding

\begin{equation}\label{Sobolev}
E_{p,q}^{s}\subset C_{b}^{s-\frac{d}{p}}
\end{equation}

if $s> \frac{d}{p}$, see p. 203 of \cite{Triebel}.

We denoted $C_{b}(\R^{d}; \R^{d})$  is the space of bounded continuous $\R^{d}$-valued functions on $\R^{d}$.

\subsection{Definition of solution }

We say that $X_{t}^{i,N}$ solve the    system (\ref{xk}) if it  is solution of the system
\begin{align}\label{VN}
\left\{
\begin{array}{lc}
d X_{t}^{k,N}=V_{t}^{k,N}dt
\\[0.3cm]
d V_{t}^{k,N}=-\nabla\left(S_{t}^{N}\ast\phi_{N}\right)(X_{t}^{k,N})dt+  \sigma (t, X_{t}^{k,N})  d Y_{t}\quad k=1,\cdots, N
\end{array}
\right.
\end{align}

where the last integral is taken in the Young sense.

  \begin{definition}
	Let  $\varrho_{0}\in E_{p,q}^{s}\cap L^{1}(\R^{d})$ with $\varrho_{0}>0$, and $\upsilon_{0}\in  E_{p,q}^{s}$ with   $s>\frac{d}{p}+2$.
Then  $(\varrho, \upsilon)$   is called a solution of (\ref{eq2}) if the following conditions are satisfied

\begin{enumerate}
    \item[(i)] $\left(\varrho, \upsilon\right)\in C([0,T], E_{p,q}^{s} ) \times  C^{\alpha}([0,T], E_{p,q}^{s} )$

\item[(ii)]  holds

$$
\begin{aligned}
&\varrho\left(t \right)=\varrho_{0}-\int_{0}^{t} \operatorname{div}_{x}(\varrho\upsilon)\mathrm{d} s, \\
&\upsilon_{q}\left(t\right)=\upsilon_{q,0}-\int_{0}^{t} \left(\frac{1}{\varrho} \nabla^{q}p+\upsilon  \nabla  \upsilon_{q} \right)\mathrm{d} s+\int_{0}^{t} \sigma_{q}   dY_{t}^{q}\quad q=1,\cdots, d
\end{aligned}
$$
for all $t \in[0, T]$.
\end{enumerate}

  \end{definition}

\subsection{ Technical hypothesis}

Next, we suppose that $\phi_{1}$ can be written as a convolution product 

 $$\phi_{1}=\phi_{1}^{r}\ast\phi_{1}^{r}$$

 where  $\phi_{1}^{r}\in C_{b}^{2}(\R^{d})$  and it is symmetric probability density.

\begin{definition}
For all  $q\in\{1,\cdots,d\}$  we define the function 
\begin{eqnarray*}
U_{1;\alpha}^{q}:\R^{d}&\longrightarrow&\R\\
x&\longmapsto&(-1)^{1+|\alpha|} \dfrac{x^{\alpha}}{\alpha!} \nabla^{q}\phi_{1}^{r}(x)
\end{eqnarray*}
where  $0\leq|\alpha|\leq L+1$ with $L:=\left[\frac{d+2}{2}\right]$.
\end{definition}

We assume that functions $\phi_{1}^{r}$ and $U_{1;\alpha}^{q}$ satisfy

\begin{eqnarray}\label{c1}
|\phi_{1}^{r}(x)|\leq\frac{C}{1+|x|^{d+2}}\qquad |x|\geq1,
\end{eqnarray}
\begin{eqnarray}\label{cotawildeu}
|\widetilde{U^{q}_{1;\alpha}}(\lambda)|\leq C|\widetilde{\phi_{1}^{r}}(\lambda)|,\qquad\text{para } 1\leq|\alpha|\leq L ,
\end{eqnarray}

and

\begin{eqnarray}\label{cotauj}
|U^{q}_{1;\alpha}(x)|\leq C \left(\frac{1}{1+|x|^{d+1}}\right)^{1/2},\qquad\text{para } |\alpha|=L+1\ .
\end{eqnarray}

We also assume that $\sigma\in  C^{\alpha}([0,T],  C_{b}^{1}( \R^{d}, \R^{d}))$.

\section{Result}

First we present a simple lemma. 

\begin{lemma}
We assume \eqref{c1}, then there exist $C_{0}>0$ such that
    \begin{eqnarray}\label{cotafNbeta}
\Vert f-(f\ast\phi_{N}^{r})\Vert_{\infty}&\leq&C_{0}N^{-\beta/d}\Vert\nabla f\Vert_{\infty}\qquad\text{for all}\,\,\,\,\, f\in C_{b}^{1}(\R^{d})
\end{eqnarray}
\end{lemma}

\begin{proof}
See \cite{Correa}
\end{proof}

We define  the energy

\begin{equation}\label{QN}
Q_{t}^{N}:=\frac{1}{N}\sum_{k=1}^{N}|V_{t}^{k,N}-\upsilon(X_{t}^{k,N},t)|^{2} +\Vert S_{t}^{N}\ast\phi_{N}^{r}-\varrho(.,t)\Vert_{L^{2}(\R^{d})}^{2}.
\end{equation}

\begin{theorem}\label{tomeuler}
Let $s\geq\frac{d}{p}+3$, $1< p,q< \infty$ $\alpha> \frac{1}{2}$,  $\eta > \frac{d}{2} +1$.  We assume \eqref{molli}, \eqref{c1}-\eqref{cotauj}.
Then there exist  $C_{T}>0$ such that

\begin{equation*}
 Q_{t}^{N}\leq C_{T}( Q_{0}^{N}+N^{-\beta/d})\quad\text{for all } N\in\N,
\end{equation*}

\begin{equation*}
\|S_{t}^{N}- \varrho_{t}\|_{E_{2,\hat{q}}^{-\eta}}^{2} \leq C_{T}( Q_{0}^{N}+N^{-\beta/d}) \quad\text{for all } N\in\N,
\end{equation*}

and

\begin{equation*}
  \|V_{t}^{N}- (\varrho \upsilon)_{t}\|_{E_{2,\hat{q}}^{\mm{-\eta}}} \leq C_{T} ( Q_{0}^{N} +N^{-\beta/d}) \quad\text{for all } N\in\N,
\end{equation*}

with  $\frac{1}{q} +\frac{1}{\hat{q}}=1$.

\end{theorem}

\begin{proof}
We recall that by  Sobolev embedding $\varrho, \upsilon\in C([0,T], C_{b}^{3}(\R^{d}))$, we use this regularity  during the proof. 
We first observe  that

\begin{align*}
Q_{t}^{N}=&\dfrac{1}{N}\sum_{k=1}^{N}\vert V_{t}^{k,N}\vert^{2}-\dfrac{2}{N}\sum_{k=1}^{N} V_{t}^{k,N}\cdot\upsilon(X_{t}^{k,N},t)\\
&+\dfrac{1}{N}\sum_{k=1}^{N}\vert\upsilon(X_{t}^{k,N},t)\vert^{2}+\Vert S_{t}^{N}\ast\phi_{N}^{r}\Vert_{L^{2}(\R^{d})}^{2}-2\langle  S_{t}^{N}\ast\phi_{N}^{r},\varrho(.,t)\rangle\\
&+\Vert\varrho(.,t)\Vert_{\mm{L^{2}(\R^{d})}}^{2}\\
=&\dfrac{1}{N}\sum_{k=1}^{N}\vert V_{t}^{k,N}\vert^{2}-\dfrac{2}{N}\sum_{k=1}^{N} V_{t}^{k,N}\cdot\upsilon(X_{t}^{k,N},t)\\
&+\dfrac{1}{N}\sum_{k=1}^{N}\vert\upsilon(X_{t}^{k,N},t)\vert^{2}+\dfrac{1}{N^{2}}\sum_{k,l=1}^{N}\phi_{N}(X_{t}^{k,N}-X_{t}^{l,N})\\
&-\dfrac{2}{N}\sum_{k=1}^{N}\left(\varrho(.,t)\ast\phi_{N}^{r}\right)(X_{t}^{k,N})+\Vert\varrho(.,t)\Vert_{L^{2}(\R^{d})}^{2}.
\end{align*}

We will calculate  the differential of each term in the sum of $Q_{t}^{N}$ via chain rules for the Young integral. 
By formula (\ref{Ito}) we deduce  

\begin{align*}
d\left(\dfrac{1}{N}\sum_{k=1}^{N}\vert V_{t}^{k,N}\vert^{2}\right)=&
\dfrac{1}{N}\sum_{k=1}^{N}d\left(\vert V_{t}^{k,N}\vert^{2}\right)=\dfrac{1}{N}\sum_{k=1}^{N}2V_{t,q}^{k,N} d(V_{t,q}^{k,N})\\
=&-\dfrac{2}{N}\sum_{k=1}^{N}V_{t,q}^{k,N}\nabla^{q}\left(S_{t}^{N}\ast\phi_{N}\right)(X_{t}^{k,N})dt\\
&+\dfrac{2}{N}\sum_{k=1}^{N}V_{t,q}^{k,N}  \sigma_{q}(t,X_{t}^{k,N}) d Y_{t}^{q}.
\end{align*}

By  formulas (\ref{produ}) and (\ref{Ito-W}) we have 

\begin{align*}
&d\left(-\dfrac{2}{N}\sum_{k=1}^{N} V_{t}^{k,N}\cdot\upsilon(X_{t}^{k,N},t)\right)
=-\dfrac{2}{N}\sum_{k=1}^{N}d[ V_{t,q}^{k,N}\upsilon_{q}(X_{t}^{k,N},t)]\\
=&-\dfrac{2}{N}\sum_{k=1}^{N}\upsilon_{q}(X_{t}^{k,N},t) d(V_{t,q}^{k,N})-\dfrac{2}{N}\sum_{k=1}^{N}V_{t,q}^{k,N}  d(\upsilon_{q}(X_{t}^{k,N},t))\\
=&\dfrac{2}{N}\sum_{k=1}^{N}\upsilon(X_{t}^{k,N},t)\cdot\nabla\left(S_{t}^{N}\ast\phi_{N}\right)(X_{t}^{k,N})dt+\dfrac{2}{N}\sum_{k=1}^{N}V_{t,q}^{k,N}\upsilon(X_{t}^{k,N},t)\cdot\nabla\upsilon_{q}(X_{t}^{k,N},t)dt\\
&+\dfrac{2}{N}\sum_{k=1}^{N}V_{t,q}^{k,N}\nabla^{q}\varrho(X_{t}^{k,N},t)dt-\dfrac{2}{N}\sum_{k=1}^{N}V_{t,q}^{k,N}\nabla\upsilon_{q}(X_{t}^{k,N},t)\cdot V_{t}^{k,N}dt\\
&-\dfrac{2}{N}\sum_{k=1}^{N}\upsilon_{q}(X_{t}^{k,N},t) \sigma_{q}(t,X_{t}^{k,N})   dY_{t}^{q}\\
&- \dfrac{2}{N}\sum_{k=1}^{N}\ V_{t,q}^{k,N}  \sigma_{q}(t,X_{t}^{k,N})  dY_{t}^{q}.\\
\end{align*}

 Again by  formulas (\ref{Ito}) and (\ref{Ito-W}) we obtain

\begin{align*}
&d\left(\dfrac{1}{N}\sum_{k=1}^{N}\vert\upsilon(X_{t}^{k,N},t)\vert^{2}\right)=\dfrac{2}{N}\sum_{k=1}^{N}\upsilon_{q}(X_{t}^{k,N},t)\ d(\upsilon_{q}(X_{t}^{k,N},t))\\
=&-\dfrac{2}{N}\sum_{k=1}^{N}\upsilon_{q}(X_{t}^{k,N},t)\upsilon(X_{t}^{k,N},t)\cdot\nabla\upsilon_{q}(X_{t}^{k,N},t) dt -\dfrac{2}{N}\sum_{k=1}^{N}\upsilon_{q}(X_{t}^{k,N},t)\nabla^{q}\varrho(X_{t}^{k,N},t)dt\\
 &+\dfrac{2}{N}\sum_{k=1}^{N}\upsilon_{q}(X_{t}^{k,N},t) \sigma_{q}(t,X_{t}^{k,N}) dY_{t}^{q}\\
&+\dfrac{2}{N}\sum_{k=1}^{N}\upsilon_{q}(X_{t}^{k,N},t)\nabla \upsilon_{q}(X_{t}^{k,N},t)\cdot V_{N}^{k}(t)dt.
\end{align*}

By Leibnitz rule simple manipulations we can deduce

\begin{align*}
& d\left(\dfrac{1}{N^{2}}\sum_{k,l=1}^{N}\phi_{N}(X_{t}^{k,N}-X_{t}^{l,N})\right)\\
=&\dfrac{1}{N^{2}}\sum_{k,l=1}^{N}d\left(\phi_{N}(X_{t}^{k,N}-X_{t}^{l,N})\right)\\
=&\dfrac{1}{N^{2}}\sum_{k,l=1}^{N}\nabla^{q}\phi_{N}(X_{t}^{k,N}-X_{t}^{l,N}) d \left(X_{t,q}^{k,N}-X_{t,q}^{l,N} \right) \\
&+ \dfrac{1}{N^{2}}\sum_{k,l=1}^{N}\nabla^{q}\phi_{N}(X_{t}^{k,N}-X_{t}^{l,N}) \left(V_{t,q}^{k,N}-V_{t,q}^{l,N} \right)dt\\
=&\dfrac{1}{N}\sum_{k=1}^{N}\left(\dfrac{1}{N}\sum_{l=1}^{N}\nabla^{q}\phi_{N}(X_{t}^{k,N}-X_{t}^{l,N})\right)\ V_{t,q}^{k,N}dt\\
&-\dfrac{1}{N}\sum_{l=1}^{N}\left(\dfrac{1}{N}\sum_{k=1}^{N}\nabla^{q}\phi_{N}(X_{t}^{k,N}-X_{t}^{l,N})\right) V_{t,q}^{l,N}dt \\
=&\dfrac{1}{N}\sum_{k=1}^{N}\left(\dfrac{1}{N}\sum_{l=1}^{N}\nabla^{q}\phi_{N}(X_{t}^{k,N}-X_{t}^{l,K})\right) V_{t,q}^{k,N}dt\\
&-\dfrac{1}{N}\sum_{l=1}^{N}\left(-\dfrac{1}{N}\sum_{k=1}^{N}\nabla^{q}\phi_{N}(X_{t}^{l,N}-X_{t}^{k,N})\right) V_{t,q}^{l,N}dt \\
=&\dfrac{2}{N}\sum_{k=1}^{N}\left(\dfrac{1}{N}\sum_{l=1}^{N}\nabla^{q}\phi_{N}(X_{t}^{k,N}-X_{t}^{l,N})\right) V_{t,q}^{k,N}dt\\
=&\dfrac{2}{N}\sum_{k=1}^{N}\left(\dfrac{1}{N}\sum_{l=1}^{N}\nabla\phi_{N}(X_{t}^{k,N}-X_{t}^{l,N})\right)\cdot V_{t}^{k,N}dt\\
=&\dfrac{2}{N}\sum_{k=1}^{N}V_{N}^{k}\cdot\nabla\left(S_{t}^{N}\ast\phi_{N}\right)(X_{t}^{k,N})dt\\
\end{align*}
By Leibnitz rule  we get  

\begin{align*}
d\left(-\dfrac{2}{N}\sum_{k=1}^{N}\left(\varrho(.,t)\ast\phi_{N}^{r}\right)(X_{t}^{k,N})\right)
=&-\dfrac{2}{N}\sum_{k=1}^{N}d\left(\left(\varrho(.,t)\ast\phi_{N}^{r}\right)(X_{t}^{k,N})\right)\\
=&-\dfrac{2}{N}\sum_{k=1}^{N}d\left(\int\varrho(x,t)\phi_{N}^{r}(x-X_{t}^{k,N})dx\right)\\
=&-\dfrac{2}{N}\sum_{k=1}^{N}V_{t}^{k,N}\cdot\nabla\left(\varrho(.,t)\ast\phi_{N}^{r}\right)(X_{t}^{k,N})dt\\
&-\dfrac{2}{N}\sum_{k=1}^{N}\left( d\varrho(.,t)\ast\phi_{N}^{r}\right)(X_{t}^{k,N})dt\\
=&-\dfrac{2}{N}\sum_{k=1}^{N}V_{t}^{k,N}\cdot\nabla\left(\varrho(.,t)\ast\phi_{N}^{r}\right)(X_{t}^{k,N})dt\\
&+\dfrac{2}{N}\sum_{k=1}^{N}\left(\operatorname{div}_{x}(\varrho(\cdot,t)\upsilon(\cdot,t))\ast\phi_{N}^{r}\right)(X_{t}^{k,N})dt.
\end{align*}

and 
\begin{align*}
d\left(\Vert\varrho(.,t)\Vert_{L^{2}(\R^{d})}^{2}\right)=&-2\langle\varrho(\cdot,t),\operatorname{div}_{x}(\varrho(\cdot,t)\upsilon(\cdot,t))\rangle dt.
\end{align*}

Adding all terms  we have

\begin{align*}
d(Q_{t}^{N})=&-\dfrac{2}{N}\sum_{k=1}^{N}V_{t,q}^{k,N}\nabla^{q}\left(S_{t}^{N}\ast\phi_{N}\right)(X_{t}^{k,N})dt\\
&+\dfrac{2}{N}\sum_{k=1}^{N}V_{t,q}^{k,N}  \sigma_{q}(t,X_{t}^{k,N})  d Y_{t}^{q}\\
&+\dfrac{2}{N}\sum_{k=1}^{N}\upsilon(X_{t}^{k,N},t)\cdot\nabla\left(S_{t}^{N}\ast\phi_{N}\right)(X_{t}^{k,N})dt\\
&+\dfrac{2}{N}\sum_{k=1}^{N}V_{t,q}^{k,N}\left(\upsilon(X_{t}^{k,N},t)\cdot\nabla\upsilon_{q}(X_{t}^{k,N},t)+\nabla^{q}\varrho(X_{t}^{k,N},t)\right)dt\\
&-\dfrac{2}{N}\sum_{k=1}^{N}V_{t,q}^{k,N}\nabla\upsilon_{q}(X_{t}^{k,N},t)\cdot V_{t}^{k,N}+\dfrac{2}{N}\sum_{k=1}^{N}\upsilon_{q}(X_{t}^{k,N},t) \sigma_{q}(t,X_{t}^{k,N})  dY_{t}^{q}\\
&- \dfrac{2}{N}\sum_{k=1}^{N} V_{t,q}^{k,N}  \sigma_{q}(t,X_{t}^{k,N}) dY_{t}^{q}\\
&-\dfrac{2}{N}\sum_{k=1}^{N}\upsilon_{q}(X_{t}^{k,N},t)\left(\upsilon(X_{t}^{k,N},t)\cdot\nabla\upsilon_{q}(X_{t}^{k,N},t)+\nabla^{q}\varrho(X_{t}^{k,N},t)\right)dt\\
&- \dfrac{2}{N}\sum_{k=1}^{N}\upsilon_{q}(X_{t}^{k,N},t) \sigma_{q}(t,X_{t}^{k,N}) dY_{t}^{q}- \dfrac{2}{N}\sum_{k=1}^{N}\upsilon_{q}(X_{t}^{k,N},t) \sigma_{q}(t,X_{t}^{k,N}) dY_{t}^{q}
\end{align*}

\begin{align*}
\quad\qquad\quad& +\dfrac{2}{N}\sum_{k=1}^{N}\upsilon_{q}(X_{t}^{k,N},t)\nabla\upsilon_{q}(X_{t}^{k,N},t)\cdot V_{t}^{k,N}dt\\
&+\dfrac{2}{N}\sum_{k=1}^{N}V_{t,q}^{k,N}\nabla^{q}\left(S_{t}^{N}\ast\phi_{N}\right)(X_{t}^{k,N})dt\\
&-\dfrac{2}{N}\sum_{k=1}^{N}V_{t}^{k,N}\cdot\nabla\left(\varrho(.,t)\ast\phi_{N}^{r}\right)(X_{t}^{k,N})dt\\
&+\dfrac{2}{N}\sum_{k=1}^{N}\left(\operatorname{div}_{x}(\varrho(\cdot,t)\upsilon(\cdot,t))\ast\phi_{N}^{r}\right)(X_{t}^{k,N})dt-2\langle\varrho(\cdot,t),\operatorname{div}_{x}(\varrho(\cdot,t)\upsilon(\cdot,t))\rangle dt
\end{align*}

 Doing obvious cancellations we have 
\begin{align*}
&d(Q_{t}^{N})\\
=&\dfrac{2}{N}\sum_{k=1}^{N}\upsilon(X_{t}^{k,N},t)\cdot\nabla\left(S_{t}^{N}\ast\phi_{N}\right)(X_{t}^{k,N})dt+\dfrac{2}{N}\sum_{k=1}^{N}V_{t,q}^{k,N}\upsilon(X_{t}^{k,N},t)\cdot\nabla\upsilon_{q}(X_{t}^{k,N},t)dt\\
&+\dfrac{2}{N}\sum_{k=1}^{N}V_{t}^{k,N}\cdot\nabla\varrho(X_{t}^{k,N},t)dt-\dfrac{2}{N}\sum_{k=1}^{N}V_{t,q}^{k,N}\nabla\upsilon_{q}(X_{t}^{k,N},t)\cdot V_{t}^{k,N} dt\\
&-\dfrac{2}{N}\sum_{k=1}^{N}\upsilon_{q}(X_{t}^{k,N},t)\upsilon(X_{t}^{k,N},t)\cdot\nabla\upsilon_{q}(X_{t}^{k,N},t)dt-\dfrac{2}{N}\sum_{k=1}^{N}\upsilon(X_{t}^{k,N},t)\cdot\nabla\varrho(X_{t}^{k,N},t)dt\\
&+\dfrac{2}{N}\sum_{k=1}^{N}\upsilon_{q}(X_{t}^{k,N},t)\nabla\upsilon_{q}(X_{t}^{k,N},t)\cdot V_{t}^{k,N}dt-\dfrac{2}{N}\sum_{k=1}^{N}V_{t}^{k,N}\cdot\nabla\left(\varrho(.,t)\ast\phi_{N}^{r}\right)(X_{t}^{k,N})dt\\
&+\dfrac{2}{N}\sum_{k=1}^{N}\left(\operatorname{div}_{x}(\varrho(\cdot,t)\upsilon(\cdot,t))\ast\phi_{N}^{r}\right)(X_{t}^{k,N})dt-2\langle\varrho(\cdot,t),\operatorname{div}_{x}(\varrho(\cdot,t)\upsilon(\cdot,t))\rangle dt\\
=&\sum_{j=1}^{10}   D_{N}(j,t).
\end{align*}

We observe that all term involving Young differential (integral) have been canceled. In the rest of the proof we follow 
line by line  \cite{Correa} and  \cite{Oes}. 

From the definition of the empirical measure and convolution,  and integration by part we obtain 
\mm{
\begin{align*}
&D_{N}(9,t)=2\left\langle S_{t}^{N},\left(\left(\operatorname{div}_{x}\left(\varrho(.,t)\upsilon(.,t)\right)\right)\ast\phi_{N}^{r}\right)(.)\right\rangle\\
=&-2\left\langle\nabla(S_{t}^{N}\ast\phi_{N}^{r})(.),\varrho(.,t)\upsilon(.,t)\right\rangle.
\end{align*}
}
From the definition of the empirical measure and integration by parts we deduce

\begin{align}
A_{N}(1,t)&:=D_{N}(1,t)+D_{N}(6,t)+D_{N}(9,t)+D_{N}(10,t)\nonumber\\
&\,\,=\dfrac{2}{N}\sum_{k=1}^{N} \upsilon(X_{t}^{k,N},t)\cdot\nabla(S_{t}^{N}\ast\phi_{N})(X_{t}^{k,N})-\dfrac{2}{N}\sum_{k=1}^{N}\upsilon(X_{t}^{k,N},t)\cdot\nabla\varrho(,t)\nonumber\\
&\,\,\,\,\,\,\,\,\,-2\left\langle\nabla(S_{t}^{N}\ast\phi_{N}^{r})(.),\varrho(.,t)\upsilon(.,t)\right\rangle-2\left\langle\varrho(.,t),\operatorname{div}_{x}(\varrho(.,t)\upsilon(.,t))\right\rangle\nonumber\\
&\,\,=2\left\langle S_{t}^{N},\nabla(S_{t}^{N}\ast\phi_{N})(.)\cdot\upsilon(,t)\right\rangle-2\left\langle S_{t}^{N},\nabla\varrho(,t)\cdot\upsilon(.,t)\right\rangle\nonumber\\
&\,\,\,\,\,\,\,\,\,-2\left\langle\varrho(.,t),\nabla(S_{t}^{N}\ast\phi_{N}^{r})(.)\cdot\upsilon(.,t)\right\rangle+2\left\langle\varrho(.,t),\nabla\varrho(.,t)\cdot\upsilon(.,t)\right\rangle.\label{AN1}
\end{align}

 We observe  
\begin{align}\label{AN2}
A_{N}(2,t):=&D_{N}(2,t)+D_{N}(4,t)+D_{N}(7,t)+D_{N}(5,t)\nonumber\\
=&\dfrac{2}{N}\sum_{k=1}^{N}V_{t,q}^{k,N}\upsilon(X_{t}^{k,N},t)\cdot\nabla\upsilon_{q}(X_{t}^{k,N},t)-\dfrac{2}{N}
\sum_{k=1}^{N}V_{t,q}^{k,N}\nabla\upsilon_{q}(X_{t}^{k,N},t) \cdot V_{t}^{k,N}\nonumber\\
&+\dfrac{2}{N}\sum_{k=1}^{N}\upsilon_{q}(X_{t}^{k,N},t)\nabla\upsilon_{q}(X_{t}^{k,N},t) \cdot V_{t}^{k,N}\nonumber\\
&-\dfrac{2}{N}
\sum_{k=1}^{N}\upsilon_{q}(X_{t}^{k,N},t)\upsilon(X_{t}^{k,N},t)\cdot\nabla\upsilon_{q}(X_{t}^{k,N},t).
\end{align}
Let $C$ be a constant that changes from line to line. Conveniently grouping and applying Young inequality we can deduce
\begin{align}\label{cotaAN2}
|A_{N}(2,t)|=&\dfrac{2}{N}\bigg|\sum_{k=1}^{N}V_{t,q}^{k,N}\upsilon_{q^{\prime}}(X_{t}^{k,N},t)\nabla^{q^{\prime}}\upsilon_{q}(X_{t}^{k,N},t)-V_{t,q}^{k,N}\nabla^{q^{\prime}}\upsilon_{q}(X_{t}^{k,N},t) 
V_{t,q^{\prime}}^{k,N}\nonumber\\
+&\upsilon_{q}(X_{t}^{k,N},t)\nabla^{q^{\prime}}\upsilon_{q}(X_{t}^{k,N},t) V_{t,q^{\prime}}^{k,N}-\upsilon_{q}(X_{t}^{k,N},t)\upsilon_{q^{\prime}}(X_{t}^{k,N},t)\nabla^{q^{\prime}}\upsilon_{q}
(X_{t}^{k,N},t)\bigg|\nonumber\\
=&\frac{2}{N}\bigg|\sum_{k=1}^{N}-\nabla^{q^{\prime}}\upsilon_{q}(X_{t}^{k,N},t)\left(V_{t,q}^{k,N}-\upsilon_{q}(X_{t}^{k,N},t)\right)\left(V_{t,q^{\prime}}^{k,N}-\upsilon_{q^{\prime}}(X_{t}^{k,N},t)\right)\bigg|\nonumber\\
\leq& C\frac{1}{N}\sum_{q,q^{\prime}=1}^{d}\sum_{k=1}^{N}\Big|V_{t,q}^{k,N}-\upsilon_{q}(X_{t}^{k,N},t)\Big|\thinspace\Big|V_{t,q^{\prime}}^{k,N}-\upsilon_{q^{\prime}}(X_{t}^{k,N},t)\Big|\nonumber\\
\leq& C\frac{1}{N}\sum_{k=1}^{N}\Big|V_{t}^{k,N}-\upsilon(X_{t}^{k,N},t)\Big|^{2}.
\end{align}
Now, we observe 
\begin{align}\label{AN3}
A_{N}(3,t):=& D_{N}(3,t)+D_{N}(8,t)\nonumber\\
=&\dfrac{2}{N}\sum_{k=1}^{N}V_{t}^{k,N}\cdot\nabla\varrho(X_{t}^{k,N},t)-\dfrac{2}{N}\sum_{k=1}^{N}V_{t}^{k,N}\cdot\nabla\left(\varrho(.,t)\ast\phi_{N}^{r}\right)(X_{t}^{k,N})\nonumber\\
=&\dfrac{2}{N}\sum_{k=1}^{N}V_{t}^{k,N}\cdot\left(\nabla\varrho(X_{t}^{k,N},t)-\nabla\left(\varrho(.,t)\ast\phi_{N}^{r}\right)(X_{t}^{k,N})\right).
\end{align}

By H\"older inequality   we obtain
\begin{align}\label{cotaAN3}
|A_{N}(3,t)|=&\frac{2}{N}\left|\sum_{k=1}^{N}V_{t}^{k,N}\cdot\left(\nabla\varrho(X_{t}^{k,N},t)-\nabla\left(\varrho(.,t)\ast\phi_{N}^{r}\right)(X_{t}^{k,N})\right)\right|\nonumber\\
\leq&\frac{2}{N}\sum_{k=1}^{N}\left|V_{t}^{k,N}\cdot\left(\nabla\varrho(X_{t}^{k,N},t)-\nabla\left(\varrho(.,t)\ast\phi_{N}^{r}\right)(X_{t}^{k,N})\right)\right|\nonumber\\
\leq&\frac{2}{N}\sum_{k=1}^{N}\left|V_{t}^{k,N}\right|\left|\left(\nabla\varrho(X_{t}^{k,N},t)-\nabla\left(\varrho(.,t)\ast\phi_{N}^{r}\right)(X_{t}^{k,N})\right)\right|\nonumber\\
\leq&\frac{2}{N}\sum_{k=1}^{N}\left|V_{t}^{k,N}\right|\Vert\nabla\varrho(.,t)-\nabla\varrho(.,t)\ast\phi_{N}^{r}(.)\Vert_{\infty}\nonumber\\
=&2\Vert\nabla\varrho(.,t)-\nabla\varrho(.,t)\ast\phi_{N}^{r}(.)\Vert_{\infty}\sum_{k=1}^{N}\frac{\left|V_{t}^{k,N}\right|}{N}\nonumber\\
\leq&2\Vert\nabla\varrho(.,t)-\nabla\varrho(.,t)\ast\phi_{N}^{r}(.)\Vert_{\infty}\left(\sum_{k=1}^{N}\frac{\left|V_{t}^{k,N}\right|^{2}}{N}\right)^{1/2}\left(\sum_{k=1}^{N}\frac{1}{N}\right)^{1/2}\nonumber\\
=&2\Vert\nabla\varrho(.,t)-\nabla\varrho(.,t)\ast\phi_{N}^{r}(.)\Vert_{\infty}\left(\frac{1}{N}\sum_{k=1}^{N}\left|V_{t}^{k,N}\right|^{2}\right)^{1/2}\nonumber\\
\leq&2\left(\frac{1}{N}\sum_{k=1}^{N}\left|V_{t}^{k,N}\right|^{2}\right)^{1/2}\sum_{i=1}^{d}\Vert\nabla^{i}\varrho(.,t)-\nabla^{i}\varrho(.,t)\ast\phi_{N}^{r}(.)\Vert_{\infty}\nonumber\\
\leq&2\left(\frac{1}{N}\sum_{k=1}^{N}\left|V_{t}^{k,N}\right|^{2}\right)^{1/2}N^{-\beta/d}\sum_{i=1}^{d}C_{0}\Vert\nabla(\nabla^{i}\varrho(.,t))\Vert_{\infty}\nonumber\\
\leq&N^{-\beta/d}C\left(\frac{1}{N}\sum_{k=1}^{N}\left|\upsilon(X_{t}^{k,N},t)\right|^{2}+\frac{1}{N}\sum_{k=1}^{N}\left|V_{t}^{k,N}-\upsilon(X_{t}^{k,N},t)\right|^{2}\right)^{1/2}\nonumber\\
\leq&N^{-\beta/d}C\left(\Vert\upsilon(.,t)\Vert_{\infty}^{2}+\frac{1}{N}\sum_{k=1}^{N}\left|V_{t}^{k,N}-\upsilon(X_{t}^{k,N},t)\right|^{2}\right)^{1/2}\nonumber\\
\leq&N^{-\beta/d}C\max\{\Vert\upsilon(.,t)\Vert_{\infty},1\}\left(1+\frac{1}{N}\sum_{k=1}^{N}\left|V_{t}^{k,N}-\upsilon(X_{t}^{k,N},t)\right|^{2}\right)^{1/2}\nonumber\\
\leq&N^{-\beta/d}C\max\{\Vert\upsilon(.,t)\Vert_{\infty},1\}\left(1+\frac{1}{N}\sum_{k=1}^{N}\left|V_{t}^{k,N}-\upsilon(X_{t}^{k,N},t)\right|^{2}\right)\nonumber\\
\leq&C\left(N^{-\beta/d}+\frac{1}{N}\sum_{k=1}^{N}\left|V_{t}^{k,N}-\upsilon(X_{t}^{k,N},t)\right|^{2}\right)
\end{align}

We have that 
\begin{equation}\label{ANRN}
A_{N}(1,t)=2\langle S_{t}^{N}-\varrho(.,t),\nabla\left(\left(S_{t}^{N}-\varrho(.,t)\right)\ast\phi_{N}\right)(.)\cdot\upsilon(.,t)\rangle+R_{N}(t)
\end{equation}

where 
\begin{align*}
R_{N}(t)=&2\langle S_{t}^{N}-\varrho(.,t),\nabla\left(\varrho(.,t)\ast\phi_{N}^{r}-\varrho(.,t)\right)(.)\cdot\upsilon(.,t) \rangle\\
&+2\langle \varrho(.,t),\nabla\left(S_{t}^{N}\ast\phi_{N}-S_{t}^{N}\ast\phi_{N}^{r}\right)(.)\cdot\upsilon(.,t)\rangle.
\end{align*}
By integration by parts and simple calculations we obtain
\begin{align}\label{RN}
R_{N}(t)=&  2\langle S_{t}^{N}-\varrho(.,t),\nabla\left(\varrho(.,t)\ast\phi_{N}^{r}-\varrho(.,t)\right)(.)\cdot\upsilon(.,t)\rangle\nonumber\\
 &+2\langle \varrho(.,t),\nabla\left(S_{t}^{N}\ast\left(\phi_{N}-\phi_{N}^{r}\right)\right)(.)\cdot\upsilon(.,t)\rangle\nonumber\\
=&2\langle S_{t}^{N}-\varrho(.,t),\nabla\left(\varrho(.,t)\ast\phi_{N}^{r}-\varrho(.,t)\right)(.)\cdot\upsilon(.,t)\rangle\nonumber\\
 &-2\langle \operatorname{div}_{x}(\varrho(\cdot,t)\upsilon(.,t)),\left(S_{t}^{N}\ast\left(\phi_{N}-\phi_{N}^{r}\right)\right)(\cdot)\rangle\nonumber\\
=&2\langle S_{t}^{N}-\varrho(.,t),\nabla\left(\varrho(.,t)\ast\phi_{N}^{r}-\varrho(.,t)\right)(.)\cdot\upsilon(.,t)\rangle\nonumber\\
 &-2\langle S_{t}^{N},\left(\operatorname{div}_{x}(\varrho(\cdot,t)\upsilon(.,t))\ast\left(\phi_{N}-\phi_{N}^{r}\right)\right)(\cdot)\rangle\nonumber\\
=&2\langle S_{t}^{N}-\varrho(.,t),\nabla\left(\varrho(.,t)\ast\phi_{N}^{r}-\varrho(.,t)\right)(.)\cdot\upsilon(.,t)\rangle\nonumber\\
&+2\langle S_{t}^{N},\left(\operatorname{div}_{x}(\varrho(\cdot,t)\upsilon(.,t))\ast\left(\phi_{N}^{r}-\phi_{N}\right)\right)(\cdot)\rangle.
\end{align}

Since   $S_{t}^{N},\varrho$  are probability densities  and \eqref{cotafNbeta} we obtain
\begin{align}\label{cotaRN}
|R_{N}(t)|\leq&2|\langle S_{t}^{N}-\varrho(.,t),\nabla\left(\varrho(.,t)\ast\phi_{N}^{r}-\varrho(.,t)\right)(.)\cdot\upsilon(.,t)\rangle|\nonumber\\
&+2\langle S_{t}^{N},|\left(\operatorname{div}_{x}(\varrho(\cdot,t)\upsilon(.,t))\ast\left(\phi_{N}^{r}-\phi_{N}\right)\right)(\cdot)|\rangle\nonumber\\
\leq&2\langle S_{t}^{N}+\varrho(.,t),1\rangle N^{-\beta/d}\sum_{i=1}^{d}C_{0}\Vert\nabla(\nabla^{i}\varrho(.,t))\Vert_{\infty}\Vert\upsilon(.,t)\Vert_{\infty}\nonumber\\
&+2N^{-\beta/d}C_{0}\Vert\nabla\left(\operatorname{div}_{x}(\varrho(\cdot,t)\upsilon(.,t))\ast\phi_{N}^{r}\right)\Vert_{\infty}\nonumber\\
=&2C N^{-\beta/d}\Vert\upsilon(.,t)\Vert_{\infty}+2N^{-\beta/d}C_{0}\Vert\left[\nabla\left(\operatorname{div}_{x}(\varrho(\cdot,t)\upsilon(.,t))\right)\right]\ast\phi_{N}^{r}\Vert_{\infty}\nonumber\\
\leq&2C N^{-\beta/d}\Vert\upsilon(.,t)\Vert_{\infty}+2N^{-\beta/d}C_{0}d\Vert\nabla\left(\operatorname{div}_{x}(\varrho(\cdot,t)\upsilon(.,t))\right)\Vert_{\infty}\Vert\phi_{N}^{r}\Vert_{1}\nonumber\\
=&2C N^{-\beta/d}\Vert\upsilon(.,t)\Vert_{\infty}+2N^{-\beta/d}C_{0}d\Vert\nabla\left(\operatorname{div}_{x}(\varrho(\cdot,t)\upsilon(.,t))\right)\Vert_{\infty}\Vert\phi_{1}^{r}\Vert_{1}\nonumber\\
\leq&CN^{-\beta/d}.
\end{align}
We follow  \cite{Oes} and \cite{Correa} the 
estimation  
\begin{align}
|A_{N}(1,t)-R_{N}(t)|&=2|\langle S_{t}^{N}-\varrho(.,t),\nabla\left(\left(S_{t}^{N}-\varrho(.,t)\right)\ast\phi_{N}\right)(.)\cdot\upsilon(.,t)\rangle|\nonumber\\
&\leq C\left(\Vert S_{t}^{N}\ast\phi_{N}^{r}-\varrho(.,t)\Vert_{L^{2}(\R^{d})}^{2}+N^{-\beta/d}\right).\label{restoA1}
\end{align}

From  \eqref{cotaAN2}, \eqref{cotaAN3}, \eqref{cotaRN}, \eqref{restoA1}  we conclude that there exist   $C>0$ such that 
\begin{align*}
 Q_{t}^{N}  \leq&   Q_{0}^{N} + C\int_{0}^{t}  \left(  Q_{s}^{N}+N^{-\beta/d}\right)ds\qquad0\leq t\leq T
\end{align*}

Then by Gronwall lemma we conclude 
\begin{eqnarray}
  Q_{t}^{N}\leq C( Q_{0}^{N} +N^{-\beta/d})\qquad\text{for all}\,\,\,\,\, N\in\N. 
\end{eqnarray}

Finally we will show the convergence  of  $S_{t}^{N}$ e $V_{t}^{N}$ to  $\upsilon$ e $\varrho \upsilon$  respectively.
 
\noindent  Let  $f\in E_{2,q}^{s}$,  we have 
\begin{align}
|\langle S_{t}^{N},f\rangle &-\langle\varrho(.,t),f \rangle|\nonumber\\
&=|\langle S_{t}^{N},f-f\ast\phi_{N}^{r}+f\ast\phi_{N}^{r}\rangle-\langle\varrho(.,t),f \rangle|\nonumber\\
&=|\langle S_{t}^{N},f-f\ast\phi_{N}^{r}\rangle+\langle S_{t}^{N}\ast\phi_{N}^{r},f\rangle-\langle\varrho(.,t),f \rangle|\nonumber\\
&=|\langle S_{t}^{N},f-f\ast\phi_{N}^{r}\rangle+\langle S_{t}^{N}\ast\phi_{N}^{r}-\varrho(.,t),f \rangle|\nonumber\\
&\leq|\langle S_{t}^{N},f-f\ast\phi_{N}^{r}\rangle|+|\langle S_{t}^{N}\ast\phi_{N}^{r}-\varrho(.,t),f \rangle|\nonumber\\
&\leq|\langle S_{t}^{N},1\rangle|\left\Vert f-f\ast\phi_{N}^{r}\right\Vert_{\infty}+|\langle S_{t}^{N}\ast\phi_{N}^{r}-\varrho(,t),f \rangle|\nonumber\\
&\leq\left\Vert f-f\ast\phi_{N}^{r}\right\Vert_{\infty}+\left\Vert f\right\Vert_{L^{2}(\R^{d})}\left\Vert S_{t}^{N}\ast\phi_{N}^{r}-\varrho(.,t)\right\Vert_{L^{2}(\R^{d})}\nonumber\\
&\leq CN^{-\beta/d}\left\Vert\nabla f\right\Vert_{\infty}+\left\Vert f\right\Vert_{L^{2}(\R^{d})}\left\Vert S_{t}^{N}\ast\phi_{N}^{r}-\varrho(.,t)\right\Vert_{L^{2}(\R^{d})}.\label{limXNvarr}
\end{align}

\mm{
We observe that

\begin{equation}\label{incl}
\left\Vert    f\right\Vert_{L^{2}(\R^{d})}  \leq  \left\Vert f\right\Vert_{ E_{2,q}^{s}}.
\end{equation}

By (\ref{Sobolev}) we have 
\begin{equation}\label{sob}
\left\Vert\nabla f\right\Vert_{\infty}\leq C  \left\Vert\nabla f\right\Vert_{ E_{2,q}^{s}}.
\end{equation}

From (\ref{limXNvarr}),   (\ref{incl})  and  (\ref{sob}) we deduce 

}

\[
  \|S_{t}^{N}- \varrho_{t} \|_{E_{2,\hat{q}}^{-s}}^{2} \leq  C_{t} (N^{-\beta/ d} +  Q_{0}^{N}). 
\]

\noindent Let $g\in E_{2,q}^{s}$, we obtain 
\begin{align}
&\left|\left\langle V_{t}^{N}-\varrho(.,t)\upsilon(.,t),g\right\rangle\right|\nonumber\\
&\quad=\left|\frac{1}{N}\sum_{k=1}^{N}V_{t}^{k,N}\cdot g(X_{t}^{k,N})-\left\langle\varrho(.,t),\upsilon(.,t)\cdot g\right\rangle\right|\nonumber\\
&\quad=\left|\frac{1}{N}\sum_{k=1}^{N}(V_{t}^{k,N}-\upsilon(X_{t}^{k,N},t)\cdot g(X_{t}^{k,N})+\left\langle S_{t}^{N}-\varrho(.,t),\upsilon(.,t)\cdot g\right\rangle\right|\nonumber\\
&\quad\leq\frac{1}{N}\sum_{k=1}^{N}\left|V_{t}^{k,N}-\upsilon(X_{t}^{k,N},t)\right|\left| g(X_{t}^{k,N})\right|\nonumber+\left|\left\langle S_{t}^{N}-\varrho(.,t),\upsilon(.,t)\cdot g\right\rangle\right|\nonumber\\
&\quad\leq\Vert g\Vert_{\infty}\frac{1}{N}\sum_{k=1}^{N}\left|V_{t}^{k,N}-\upsilon(X_{t}^{k,N},t)\right|+|\langle S_{t}^{N}-\varrho(.,t),\upsilon(.,t)\cdot g\rangle|\nonumber\\
&\quad\leq C\Vert g\Vert_{\infty}\left(\frac{1}{N}\sum_{k=1}^{N}|V_{t}^{k,N}-\upsilon(X_{t}^{k,N},t)|^{2} \right)^{1/2}+|\langle S_{t}^{N}-\varrho(.,t),\upsilon(.,t)\cdot g\rangle |.\label{convV}
\end{align}

From  (\ref{convV})  and  (\ref{Sobolev}) we obtain

\[
\|V_{t}^{N}- (\varrho \upsilon)_{t}\|_{E_{2, \hat{q}}^{-s}}^{2} \leq  C_{T} (N^{-\beta/d} + Q_{0}^{N}). 
\]

\end{proof}

\section*{Acknowledgements}
Author Jesus Correa  has received research grants from CNPq
through the grant  141464/2020-8. Author Juan D. Londoño has received grants from the Coordenação de Aperfeiçoamento de Pessoal de Nível Superior-Brasil (CAPES)-Finance Code 001 Author Christian Olivera is partially supported by FAPESP by the grant  2020/04426-6,  by  CNPq by the grant $422145/2023-8$
 and by FAPESP-ANR by the grant Stochastic and Deterministic Analysis for Irregular Models-2022/03379-0.

\end{document}